\documentclass[a4paper]{amsart}

\usepackage{ascmac,amssymb,amsmath,amsthm}
\usepackage{mathtools}
\usepackage{graphicx,color}
\usepackage{tikz}
\usepackage{thmtools}
\usepackage{stmaryrd}
\usepackage[hidelinks]{hyperref}

\usetikzlibrary{cd}
\newcommand{\proj}{\mathbb{P}^{1}}%projective space
\newcommand{\Z}{\mathbb{Z}}%integer
\newcommand{\C}{\mathbb{C}}%complax number
\newcommand{\spec}[1]{\mathrm{Spec}(#1)}
\newcommand{\spe}[1]{\mathrm{Spec}\, #1}
\newcommand{\tr}[2]{\mathrm{tr}(#1 \colon#2 )}
\newcommand{\per}[2]{\mathrm{Per}_{#1}(#2)}

\newcommand{\dzeta}[2]{Z_{#1}(#2 ; t )}
\newcommand{\lzeta}[2]{L_{#1}(#2 ; t )}
\renewcommand{\epsilon}{\varepsilon}
\newcommand{\sheaf}{\Omega_{X}^{\otimes m}}
\newcommand{\mult}[2]{\lambda_{#2}(#1)}
\newcommand{\fixpt}[1]{\mathrm{Fix}(#1)}

\newcommand{\cohomo}[2]{H^{#1}(\proj_K,\Omega_{\proj_K}^{\otimes #2})}
\newcommand{\dyndet}[1]{\det(1-tD_{#1}^1(\phi))}%(\phi)\colon \cohom{#1})}
\newcommand{\pgl}{\mathrm{PGL}_2(K)}
\newcommand{\dform}[1]{\left(\frac{dz}{z}\right)^{\otimes #1}}
\newcommand{\lamz}{\lambda_{0}}
\newcommand{\lami}{\lambda_{\infty}}
\newcommand{\Fz}{F_{0}}
\newcommand{\Fi}{F_{\infty}}
\theoremstyle{plain}
\newtheorem{thm}{Theorem}[section]
\newtheorem{prop}[thm]{Proposition}
\newtheorem{lem}[thm]{Lemma}
\newtheorem{cor}[thm]{Corollary}
\theoremstyle{definition}
\newtheorem{defin}[thm]{Definition}
\newtheorem{ex}[thm]{Example}
\newtheorem{rem}[thm]{Remark}
\title{The rationality of dynamical zeta functions and Woods Hole fixed point formula}
\author{Kohei Takehira}
\address{Mathematical Institute, Graduate School of Science, Tohoku University, 6-3 Aramakiaza, Aoba, Sendai, Miyagi 980-8578, Japan.}
\email{kohei.takehira.p5@dc.tohoku.ac.jp}

\begin{document}
\maketitle 

\begin{abstract}
For one variable rational function $\phi\in K(z)$ over a field $K$, we can define a discrete dynamical system by regarding $\phi$ as a self morphism of $\proj_K$. Hatjispyros and Vivaldi defined a dynamical zeta function for this dynamical system using multipliers of periodic points, that is, an invariant which indicates the local behavior of dynamical systems. In this paper, we prove the rationality of dynamical zeta functions of this type for a large class of rational functions $\phi\in K(z)$. The proof here relies on Woods Hole fixed point formula and some basic facts on the trace of a linear map acting on cohomology of a coherent sheaf on $\proj_{K}$.
\end{abstract}

\tableofcontents

%% Section 1 %%%%%%%%%%%%%%%%%%%%%%%%%%%%%%%%%%%%%%%%%%%%%%%%%%%%%%%%%%%%%%
%	Introduction and Acknowledgement
%	
%
%
%%%%%%%%%%%%%%%%%%%%%%%%%%%%%%%%%%%%%%%%%%%%%%%%%%%%%%%%%%%%%%%%%%%%%%%%%
\section{Introduction}

In this paper,  we study the rationality of the dynamical zeta function introduced by Hatjispyros and Vivaldi in \cite{Hatjispyros-Vivaldi}. In this section, we introduce the dynamical zeta function $\dzeta{m}{\phi}$ without explaining details and state our main result.

Let $K$ be an algebraically closed field with characteristic 0. For a rational function $\phi\in K(z)$ and a non-negative integer $m\in\Z_{\ge 0}$, we define the \textit{dynamical zeta function} as
\[\dzeta{m}{\phi}=\exp\left(\sum_{n=1}^{\infty}\frac{t^n}{n}\sum_{x\in \per{n}{\phi}}\mult{\phi^n}{x}^m\right)\in K\llbracket t\rrbracket.\]
$\per{n}{\phi}=\{x\in \proj_{K} : \phi^{n}(x)=x\}=\fixpt{\phi^n}$ is the set of periodic points of period $n$ and 
	\begin{align*}
	\mult{\phi}{x}=
		\begin{cases}
		\phi'(x) & \text{if } x\neq \infty \\
		\psi'(0) & \text{if } x=\infty
		\end{cases},
	\end{align*}
where $\psi(z)=1/\phi(1/z)$, is the \textit{multiplier} of the fixed point $x\in \fixpt{\phi}$.

For $m=0$, $\dzeta{0}{\phi}$ is Artin-Mazur zeta function defined by Artin and Mazur in \cite{Artin-Mazur}.
\[\dzeta{0}{\phi}=\exp\left( \sum_{n=1}^{\infty}\frac{t^n}{n}\# \per{n}{\phi}\right).\]

Hinkkanen showed that $\dzeta{0}{\phi}$ is a rational function for $K=\C$ in \cite[Theorem 1]{Hinkkanen}. More precisely, he showed that there exists $N\in\Z_{\ge 0}$, $p_1,\dots p_N,q_1,\dots, q_N, l_1,\dots, l_N \in \Z_{\ge 0}$ such that
\[\dzeta{0}{\phi}=\frac{1}{(1-t)(1-dt)}\prod_{i=1}^{N}(1-t^{p_i q_i})^{l_i}.\]
He used techniques from complex dynamics. Lee pointed out the Hinkkanen's result is still valid for any algebraically closed field $K$ with characteristic $0$ in \cite[Theorem 1.1.]{Lee}

For $m>0$, Hatjispyros and Vivaldi showed the rationality of $\dzeta{1}{\phi}$ for a polynomial of special type $\phi(z)=z^d+c$ in \cite[Lemma 3.1.]{Hatjispyros-Vivaldi} and conjectured that $\dzeta{m}{\phi}$ is a rational function for quadratic polynomials in the same paper. Predrag Cvitanovic, Kim Hansen, Juri Rolf, and G{\'{a}}bor Vattay showed the rationality of $\dzeta{m}{\phi}$ for quadratic polynomials in \cite[Section 4.1.]{Cvitanovic}. Eremenko and Levin showed that $\dzeta{1}{\phi}$ is a rational function for each polynomial in \cite[Lemma 1]{Eremenko-Levin}.

For $m>0$, the rationality of $\dzeta{m}{\phi}$ in previous results is restricted to the case that $\phi$ is a polynomial. Our main result is the rationality of the zeta function $\dzeta{m}{\phi}$ for a rational function $\phi\in K(z)$ which is not supposed to be a polynomial.

\begin{thm}
Let $K$ be an algebraiclly closed field with characteristic $0$ and let $\phi \in K(z)$ be a  rational function of degree $\ge 2$. Assume that $\phi$ has no periodic point with multiplier 1. Then the dynamical zeta function $\dzeta{m}{\phi}$ is a rational function over $K$ for all $m\in \Z_{\ge 0}$ .
\end{thm}

This paper will be organized as follows. In Section \ref{preliminaries}, we recall some basic notions of the discrete dynamics on $\proj$ and some results on dynamical zeta functions. We devote Section \ref{woodshole} to recall the Woods Hole fixed point formula, which is the key tool of our proof for the main theorem. In Section \ref{main}, we prove the main theorem after some observation. In Section \ref{examples}, We construct some examples of rational function $\phi\in \C(z)$ that we can apply our main theorem and we give the explicit formula of $\dzeta{1}{\phi}$ for such $\phi$.

\section*{Acknowledgement}
I would like to show my greatest appreciation to Professor Nobuo Tsuzuki for his advice and helpful comments. I am deeply grateful to Professor Takuya Yamauchi for reading the manuscript carefully and telling me related works. I would like to Yuya Murakami for his technical advice. I would also like to thank Atsuhide Nagasaka, Hiroyuki Sunata, Junpei Kasama, and Masahiro Watanabe for their helpful comments. The author is supported by the WISE Program for AI Electronics, Tohoku University.

%% Section 2 %%%%%%%%%%%%%%%%%%%%%%%%%%%%%%%%%%%%%%%%%%%%%%%%%%%%%%%%%%%%%%
%	Prelimiaries
%		Periodic points and its multipliers
%		Dynamical zeta function
%
%%%%%%%%%%%%%%%%%%%%%%%%%%%%%%%%%%%%%%%%%%%%%%%%%%%%%%%%%%%%%%%%%%%%%%%%%
\section{Prelimiaries}\label{preliminaries}

In this section, we recall some basic notions on dynamical systems on $\proj$ associated with a rational function and we see some examples and results of dynamical zeta functions. Notation in this section follows Silverman's textbook \cite{Silverman}.
 
\subsection{Periodic points and its multipliers}
Let $K$ be a field. For a rational function $\phi\in K(z)$, we can regard $\phi$ as an endomorphism $\phi:\proj_K\to \proj_K$. For any integer $n\ge 0$, $\phi^n:\proj_K\to\proj_K$ denotes the \textit{n-th iterate} of $\phi$,
\begin{align*}
\phi^n = \underbrace{\phi\circ\phi\circ \cdots \circ \phi }_{n}.
\end{align*}

\begin{defin}
Let $\phi \in K(z)$ be a rational function.
	\begin{itemize}
		\item[$(1)$]A point $x\in\proj_K$ is said to be a \textit{fixed point} of $\phi$ if $\phi(x)=x$. We denote by $\fixpt{\phi}$ the set of fixed points of  $\phi$.
		\item[$(2)$]We say a point $x\in\proj_K$ is a \textit{periodic point} of period $n$ if $\phi^{n}(x)=x$. We denote by $\per{n}{\phi}=\fixpt{\phi^n}$ the set of periodic points of period $n$.
		\item[$(3)$]A periodic point $x$ is said to be \textit{of minimal period} $n$ if $x\in \per{n}{\phi}$ and $x\not\in \per{m}{\phi}$ for any $0<m<n$. We denote by $\mathrm{Per}_{d}^{**}(\phi)$ the set of periodic points of minimal period $n$.
		\item[$(4)$]The \textit{forward orbit} of $x\in\proj_K$ is defined by $\mathcal{O}_{\phi}(x)=\{\phi^{n}(x)\colon n\in\Z_{\ge 0}\}$.
	\end{itemize} 
\end{defin}

It is easy to show that we can decompose $\per{n}{\phi}$ into the disjoint union $\per{n}{\phi}=\bigcup_{d|n}\mathrm{Per}_{d}^{**}(\phi)$. Next, we define the multiplier of fixed points. Recall that the linear fractional transformation $\theta(z)=\dfrac{az+b}{cz+d}$ for $ \begin{pmatrix}a & b\\ c & d\end{pmatrix}\in\mathrm{GL}_{2}(K)$ defines an automorphism of $\proj_K$.

\begin{defin}
For $\phi\in K(z)$ and $\theta\in \pgl$ the \textit{linear conjugate} of $\phi$ by $\theta$ is the map $\theta\circ\phi\circ \theta^{-1}$.
\end{defin}

The linear conjugation of $\phi$ at $\theta$ yields the following commutative diagram.
\begin{center}
	\begin{tikzcd}
		\proj_K \arrow[d,"\theta"']\arrow[r,"\phi"]					& \proj_K \arrow[d,"\theta"]& \\
		\proj_K \arrow[r,"\theta\circ\phi\circ \theta^{-1}"']	& \proj_K.
	\end{tikzcd}
\end{center}

Given $x\in \fixpt{\phi}\setminus\{\infty\}$, we define the \textit{multiplier} $\mult{\phi}{x}$ of $\phi$ at $x$ by $\mult{\phi}{x}=\phi'(x)$. The multiplier is invariant under the linear conjugation.

\begin{prop}
Let $\phi\in K(z)$ be a rational function and $\theta\in\pgl$ be a fractional linear transformation. Then the followings hold.
	\begin{itemize}
		\item[$(1)$]$\fixpt{\theta\circ\phi\circ \theta^{-1}}=\theta(\fixpt{\phi})$.
		\item[$(2)$]For any $x\in\fixpt{\phi}$ , $\mult{\theta\circ\phi\circ \theta^{-1}}{\theta(x)}=\mult{\phi}{x}$ if $\theta(x) \neq \infty$.
	\end{itemize}
\end{prop}
\begin{proof}
The first statement is trivial. The second statement is an easy consequence of the chain rule.
\end{proof}

Using this property, we can extend the definition of the multiplier as below.

\begin{defin}
Let $\phi\in K(z)$ be a rational function. We define the \textit{multiplier} of $\phi$ at each $x\in\fixpt{\phi}$  as
	\begin{align*}
	\mult{\phi}{x}=
		\begin{cases}
		\phi'(x) & \text{if } x\neq \infty \\
		\psi'(0) & \text{if } x=\infty
		\end{cases},
	\end{align*}
where $\psi(z)=1/\phi(1/z)$, which is the linear conjugation of $\phi$ at $\theta(z)=1/z$. 
\end{defin}

\subsection{Dynamical zeta function}

\begin{defin}
Let $K$ be an algebraically closed field with characteristic 0. For a rational function $\phi\in K(z)$ and a non-negative integer $m\in\Z_{\ge 0}$, we define the \textit{dynamical zeta function} as
\[\dzeta{m}{\phi}=\exp\left(\sum_{n=1}^{\infty}\frac{t^n}{n}\sum_{x\in \per{n}{\phi}}\mult{\phi^n}{x}^m\right)\in K\llbracket t\rrbracket.\]
\end{defin}

\begin{rem}
Since the multiplier is invariant under the linear conjugation, $\dzeta{m}{\phi}$ is also invariant under the linear conjugation.
\end{rem}

In \cite[Section 2]{Hatjispyros-Vivaldi}, Hatjispyros and Vivaldi pointed out that the dynamical zeta function $\dzeta{m}{\phi}$ has the Eulerian product as follows. We denote by $\mathrm{Per}_{n}^{**}(\phi)/\sim$ the quotient set of $\mathrm{Per}_{n}^{**}(\phi)$ by the equivalent relation $x\sim y \Leftrightarrow \mathcal{O}_{\phi}(x)=\mathcal{O}_{\phi}(y)$.

\begin{thm}
\[\dzeta{m}{\phi}=\prod_{n=1}^{\infty}\prod_{x\in \mathrm{Per}_{n}^{**}(\phi)/\sim} (1-\mult{\phi^n}{x}^{m}t^{n})^{-1}\]
where the second product runs over a complete system of representatives of  $\mathrm{Per}_{n}^{**}(\phi)/\sim$.
\end{thm}

\begin{proof}
Using the chain rule, we have
	\begin{itemize}
	\item $\mult{\phi^{md}}{x}=\mult{\phi^{d}}{x}^m$ for any $x\in \per{d}{\phi}$,
	\item $\mult{\phi^n}{x}=\mult{\phi^n}{y}$ if $x\sim y$.
	\end{itemize}
Since we can decompose $\per{n}{\phi}$ into the disjoint union $\per{n}{\phi}=\bigcup_{d|n}\mathrm{Per}_{d}^{**}(\phi)$, we obtain 
\begin{align*}
\sum_{x\in \per{n}{\phi}}\mult{\phi^n}{x}^m &= \sum_{d|n}\sum_{x\in \mathrm{Per}_{d}^{**}(\phi)}\mult{\phi^n}{x}^m\\
&=\sum_{d|n}\sum_{x\in \mathrm{Per}_{d}^{**}(\phi)/\sim} d(\mult{\phi^{d}}{x}^m)^{n/d}.
\end{align*}
Therefore, 
	\begin{align*}
		\dzeta{m}{\phi}&=\exp\left(\sum_{n=1}^{\infty}\sum_{d|n}\sum_{\mathrm{Per}_{d}^{**}(\phi)/\sim} \frac{(\mult{\phi^{d}}{x}^m)^{n/d}}{n/d}t^n\right)\\
		&=\exp\left(\sum_{l=1}^{\infty}\sum_{d=1}^{\infty}\sum_{\mathrm{Per}_{d}^{**}(\phi)/\sim} \frac{(\mult{\phi^{d}}{x}^mt^d)^{l}}{l}\right)\\
		&=\exp\left(-\sum_{d=1}^{\infty}\sum_{\mathrm{Per}_{d}^{**}(\phi)/\sim}\log (1-\mult{\phi^d}{x}^mt^d)\right)\\
		&=\prod_{d=1}^{\infty}\prod_{\mathrm{Per}_{d}^{**}(\phi)/\sim}(1-\mult{\phi^d}{x}^mt^d)^{-1}.\\
	\end{align*}
\end{proof}

\begin{ex} Let $\phi(z)=z^d$ for an integer $d\ge 2$. Since $\phi^n(z)=z^{d^{n}}$, we have $\per{n}{\phi}=\{0\}\cup \mu_{d^{n}-1}=\{0\}\cup \{\zeta\in K \colon \zeta^{d^{n}-1}=1\}$. Therefore, 
	\begin{align*}
		\mult{\phi^{n}}{x}=
			\begin{cases}
			0	& (x=0) \\
			d^n & (x\in \mu_{d^{n}-1})
			\end{cases}
	\end{align*}
and 
\[\sum_{x\in\per{n}{\phi}} \mult{\phi^{n}}{x}^m=(d^n-1)d^{nm}=d^{n(m+1)}-d^{nm}.\]
 So we obtain
\[\dzeta{m}{\phi}=\frac{1-d^{m} t}{1-d^{m+1} t}.\]
\end{ex}

\begin{ex}
Let $T_{d}$ be the \textit{d-th Chebyshev polynomial} satisfying $2\cos (d\theta)=T_d(2\cos \theta)$. It is well known that $T_{d}^{n}=T_{d^n}$, $\per{n}{T_d}=\{\zeta+\zeta^{-1}\colon \zeta \in \mu_{d-1}\}\cup\{\zeta+\zeta^{-1}\colon \zeta \in \mu_{d+1}\}\cup\{\infty\}$ and
	\begin{align*}
	\mult{T_d}{x}=
		\begin{dcases}
		d & (x=\zeta+\zeta^{-1} , \zeta\in \mu_{d-1}\setminus\{\pm 1\})\\
		-d & (x=\zeta+\zeta^{-1} , \zeta\in \mu_{d+1}\setminus\{\pm 1\})\\
		d^2 & (x= \pm 2)\\
		0 & (x=\infty)
		\end{dcases}.
	\end{align*}
For details, see \cite[Section~6.2]{Silverman}. Therefore, we obtain
	\begin{align*}
	\dzeta{m}{T_d}=
		\begin{dcases}
		\frac{1-d^{m}t}{1-d^{2m}t} \frac{1}{1-d^{m+1}t} & (d:\text{even},m:\text{even})\\
		\frac{1-d^{m}t}{1-d^{2m}t} 	& (d:\text{even},m:\text{odd})\\
		\left(\frac{1-d^{m}t}{1-d^{2m}t}\right)^2 \frac{1}{1-d^{m+1}t} & (d:\text{odd},m:\text{even})\\
		\frac{1-d^{m}t}{(1-d^{2m}t)^2} & (d:\text{odd},m:\text{odd})
		\end{dcases}.
	\end{align*}
This result is firstly obtained by Hatjispyros in \cite[Theorem 1]{Hatjispyros}. 
\end{ex}

%% Section 3 %%%%%%%%%%%%%%%%%%%%%%%%%%%%%%%%%%%%%%%%%%%%%%%%%%%%%%%%%%%%%%

%% Section 3 %%%%%%%%%%%%%%%%%%%%%%%%%%%%%%%%%%%%%%%%%%%%%%%%%%%%%%%%%%%%%%
%	Woods Hole fixed point formula
%	
%
%
%%%%%%%%%%%%%%%%%%%%%%%%%%%%%%%%%%%%%%%%%%%%%%%%%%%%%%%%%%%%%%%%%%%%%%%%%
\section{Woods Hole fixed point formula}\label{woodshole}

In this section, we review the Woods Hole fixed point formula which is the key tool of our proof. This formula is also called \textit{Atiyah-Bott fixed point formula} since Atiyah and Bott proved this formula in differential geometry in \cite[Theorem A]{AtBo} and \cite[Theorem A]{AB}. A purely algebraic proof can be found in SGA5 \cite[Expos{\'{e}} III, Corollaire 6.12]{SGA5} and in \cite[Theorem A.4.]{Taelman}. For details and other applications of this formula, see \cite{SGA5}, \cite{Taelman}, \cite{Beauville}, \cite{Kond}, and \cite{Ramirez}.

Let $X$ be a Noetherian scheme over an algebraically closed field $K$, $\mathcal{F}$ and $\mathcal{G}$ be coherent sheaves on $X$, and  $\varphi: \mathcal{G}\to\mathcal{F}$ is a sheaf homomorphism. For $x\in X$, we define $\mathcal{F}(x)=\mathcal{F}_x\otimes_{\mathcal{O}_{X,x}}\mathcal{O}_{X,x}/\mathfrak{m}_x$. Then we obtain a natural $\mathcal{O}_{X,x}/\mathfrak{m}_x$-linear map $\varphi(x): \mathcal{G}(x)\to \mathcal{F}(x)$. If $f:X\to X$ is an endomorphism and $\mathcal{G}=f^{*}\mathcal{F}$, we have a map $\widetilde{\varphi}^{(p)}:H^{p}(X,\mathcal{F})\to H^{p}(X,\mathcal{F})$ satisfying
	\begin{center}
		\begin{tikzcd}
			H^{p}(X,\mathcal{F}) \arrow[rd,"\widetilde{\varphi}^{(p)}"]\arrow[d]& \\
			H^{p}(X,f^{*}\mathcal{F}) \arrow[r]& H^{p}(X,\mathcal{F})
		\end{tikzcd}.
	\end{center}
$H^{p}(X,\mathcal{F})\to H^{p}(X,f^{*}\mathcal{F})$ is the pull-back on cohomology and $H^{p}(X,f^{*}\mathcal{F})\to H^{p}(X,\mathcal{F})$ is the homomorphism induced by $\varphi:f^{*}\mathcal{F}\to\mathcal{F}$. Note that the pull-back $H^{p}(X,\mathcal{F})\to H^{p}(X,f^{*}\mathcal{F})$ is decomposed as $H^{p}(X,\mathcal{F})\to H^{p}(X,f_{*}f^{*}\mathcal{F})\to H^{p}(X,f^{*}\mathcal{F})$, where $H^{p}(X,\mathcal{F})\to H^{p}(X,f_{*}f^{*}\mathcal{F})$ is induced by the sheaf homomorphism $\mathcal{F}\to f_{*}f^{*}\mathcal{F}$. Moreover, note that if $f$ is affine, $H^{p}(X,f_{*}f^{*}\mathcal{F})\to H^{p}(X,f^{*}\mathcal{F})$ is an isomorphism \cite[III, Excercise 8.2.]{Hartshorne}.

\begin{thm}(Woods Hole Fixed Point Formula, \cite[Theorem A.4.]{Taelman})
Let $X$ be a smooth proper scheme over an algebraically closed field $K$ and let $f: X \to X$ be an endomorphism. Let $\mathcal{F}$ be a locally free $\mathcal{O}_{X}$ module of finite rank and let $\varphi: f^{*}\mathcal{F}\to \mathcal{F}$ be a homomorphism of $\mathcal{O}_{X}$ module. Assume that the graph $\Gamma_{f}\subset X\times X$ and  the diagonal $\Delta\subset X\times X$ intersect transversally in $X\times X$. Then the following identity holds.
\[\sum_{p}(-1)^{p}\tr{\widetilde{\varphi}^{(p)}}{H^{p}(X,\mathcal{F})}=\sum_{x\in \fixpt{f}}\frac{\tr{\varphi(x)}{\mathcal{F}(x)}}{\det (1-df(x)\colon \Omega_{X/K}(x))} \]
where $df:f^{*}\Omega_X\to\Omega_X$ is the differential of $f$.
\end{thm}

For a morphism $f:X\to Y$, the differential $df:f^{*}\Omega_Y\to\Omega_X$ locally comes from 
\[ \Omega_B\otimes_{B}A\to \Omega_A; db\otimes a\mapsto ad(\varphi(b)),\]
where $\spec{B}\subset X$ and $\spec{A} \subset Y$ are affine open subsets and $f$ corresponds to a ring homomorphism $\varphi: B\to A$.

The statement that $\Gamma_{f}\subset X\times X$ and  the diagonal $\Delta\subset X\times X$ intersect transversally in $X\times X$ means that $\Gamma_{f}$ and $\Delta$ meet with intersection multiplicity $1$ at each point $x\in \Gamma_{f}\cap\Delta$.

\begin{rem}
It is known that the followings are equivalent.
\begin{itemize}
\item[$(1)$]The graph $\Gamma_{f}\subset X\times X$ and  the diagonal $\Delta\subset X\times X$ intersect transversally in $X\times X$. 
\item[$(2)$]$\det (1-df(x))\neq 0$ for all $x$ satisfying $f(x)=x$.
\end{itemize}
For details, see \cite[Proposition I.1]{Beauville}. Therefore, in the case that $X=\proj_K$ and  $\phi\in K(z)$, the transversality condition equals to that $\phi$ has no fixed point with multiplier $1$.
\end{rem}

%% Section 4 %%%%%%%%%%%%%%%%%%%%%%%%%%%%%%%%%%%%%%%%%%%%%%%%%%%%%%%%%%%%%%
%	Introduction
%	
%
%
%%%%%%%%%%%%%%%%%%%%%%%%%%%%%%%%%%%%%%%%%%%%%%%%%%%%%%%%%%%%%%%%%%%%%%%%%
\section{Main theorem}\label{main}

In this section, we prove, by applying Woods Hole fixed point formula, the rationality of dynamical zeta functions $\dzeta{m}{\phi}$ for each \textit{completely transversal} rational function $\phi \in K(z)$.

%% Section 4.1. %%%%%%%%%%%%%%%%%%%%%%%%%%%%%%%%%%%%%%%%%%%%%%%%%%%%%%%%%%%%%%
\subsection{Notation}\label{preparation}

\begin{defin}
Let $K$ be an algebraiclly closed field with characteristic $0$ and let $\phi \in K(z)$ be a rational function. We say that $\phi$ is \textit{completely transversal} if $\mult{\phi^n}{x}\neq 1$ for any $n\in\Z_{\ge 0}$ and  $x\in\per{n}{\phi}$.
\end{defin}

\begin{defin}
For $\phi\in K(z)$ with no fixed point with multiplier $1$ and $m\in\Z_{\ge 0}$, we define
\[T_{m}(\phi)=\sum_{x\in \fixpt{\phi}}\frac{\mult{\phi}{x}^{m}}{1-\mult{\phi}{x}}.\]
\end{defin}

\begin{rem}\label{decomposition}
Note that 
\[ \sum_{x\in\fixpt{\phi}}\mult{\phi}{x}^{m}=T_{m}(\phi)-T_{m+1}(\phi)\]
for all rational functions $\phi \in K(z)$ with no fixed point with multiplier $1$. So we have
\[ \sum_{x\in\per{n}{\phi}}\mult{\phi^n}{x}^{m}=T_{m}(\phi^n)-T_{m+1}(\phi^n)\]
for all $n\in\Z_{>0}$ if $\phi$ is completely transversal.
\end{rem}

\begin{defin}
Let $K$ be an algebraically closed field with characteristic $0$ and let $\phi \in K(z)$ be a completely transversal rational function of degree $\ge 2$. 
We define $\lzeta{m}{\phi}\in K\llbracket t\rrbracket$ by
\[\lzeta{m}{\phi}=\exp\left( \sum_{n=1}^{\infty}\frac{t^n}{n} T_{m}(\phi^n)\right).\]
\end{defin}

\begin{rem}\label{quotient}
By Remark \ref{decomposition}, we have
\[\dzeta{m}{\phi}=\frac{\lzeta{m}{\phi}}{\lzeta{m+1}{\phi}}\]
if $\phi$ is completely transversal.
\end{rem}

We state the main theorem as below. A proof will given in Section \ref{conclution} after some observation.
\begin{thm}\label{maintheorem}
Let $K$ be an algebraiclly closed field with characteristic $0$ and let $\phi \in K(z)$ be a completely transversal rational function of degree $\ge 2$.  Then the following hold.
\[\lzeta{m}{\phi}=\dfrac{\det (1-tD_m^1(\phi))}{\det (1-tD_{m}^0(\phi))}\in K(t),\]
where $D_m^p(\phi): H^{p}(\proj_K,\Omega_{\proj_K}^{\otimes m})\to H^{p}(\proj_K,\Omega_{\proj_K}^{\otimes m})$ is the map associated to the sheaf homomorphism $(d\phi)^{\otimes m}:{\phi}^{*}\Omega_{\proj_K}^{\otimes m}\to \Omega_{\proj_K}^{\otimes m}$.
Especially, $\dzeta{m}{\phi}\in K(t)$ for any $m\in\Z_{\ge0}$. 
\end{thm}

\begin{rem}
For $m>0$, we have
\[\dzeta{m}{\phi}=\frac{\dyndet{m}}{\dyndet{m+1}}\]
since $\Omega_{\proj_K} \cong \mathcal{O}_{\proj_K}(-2)$,  $\Omega_{\proj_K}^{\otimes m} \cong \mathcal{O}_{\proj_K}(-2m)$ and 
	\begin{align*}
		\dim_K H^{i}(\proj_K,\mathcal{O}_{\proj_K}(-2m))=
			\begin{cases}
				2m-1 & (i=1)\\
					0		& (\text{otherwise}).
			\end{cases}
	\end{align*}
\end{rem}

%% Section 4.2. %%%%%%%%%%%%%%%%%%%%%%%%%%%%%%%%%%%%%%%%%%%%%%%%%%%%%%%%%%%%%%
\subsection{Connecton between multipliers and sheaf cohomologies}

\begin{lem}\label{lem_localvsgrobal}
Let $\phi \in K(z)$ be a rational function with no fixed point with multiplier $1$. Then for any positive integer $m\in\Z_{>0}$, the following holds.  
\[T_m(\phi)=\sum_{i=0}^{1}(-1)^{i}\tr{D_m^i(\phi)}{\cohomo{i}{m}},\]
where $D_m^i(\phi):\cohomo{i}{m}\to\cohomo{i}{m}$ is the $K$-linear map assosiated to the sheaf homomorphism $(d\phi)^{\otimes m}:\phi^{*}\Omega_{\proj_K}^{\otimes m}\to \Omega_{\proj_K}^{\otimes m}$.
\end{lem}

\begin{proof}
After changing a coordinate, we may assume that $\infty \not\in \fixpt{\phi}$. We apply Woods Hole fixed point formula for $X=\proj_K$, $\mathcal{F}=\sheaf$, $f=\phi$ and $\varphi=(d\phi)^{\otimes m}:\phi^{*}\sheaf\to\sheaf$ and obtain
\[\sum_{x\in\fixpt{\phi}}\frac{\tr{d\phi^{\otimes m}(x)}{\Omega_{\proj_K}^{\otimes m}(x)}}{\det (1-d\phi(x)\colon \Omega_{\proj_K}(x))}=\sum_{i=0}^{1}(-1)^{i}\tr{D_m^i(\phi)}{\cohomo{i}{m}}.\]
For $x\in\fixpt{\phi}$, we take a local parameter $t=z-x$ of $\proj_K$ at $x$. Then we have \begin{itemize}
\item $\mathcal{O}_{\proj_K,x}\cong K[z]_{(t)}$,
\item $\mathcal{O}_{\proj_K,x}/\mathfrak{m}_{x}\cong K$ via the evalutation at $t=0$, and
\item $\Omega_{\proj_K,x}\cong K[z]_{(t)}dt$.
\end{itemize}
Therefore, we have
 \[(d\phi)_{x}(dt)=d(\phi^{*}t)=d\phi(t+x)=\phi'(t+x)dt\]
and $\Omega_{\proj_K}(x)=K[z]_{(t)}dt\otimes_{K[z]_{(t)}}(K[z]_{(t)}/(t)K[z]_{(t)})\cong Kdt$ via $a(t)dt\otimes 1\mapsto a(0)dt$. So
\[d\phi(x)(dt)=\phi'(x)dt=\mult{\phi}{x}dt.\]
This yields that $d\phi(x)=\mult{\phi}{x} \mathrm{id}$. So we have
	\begin{align*}
		T_m(\phi)&=\sum_{x\in \fixpt{\phi}}\frac{\mult{\phi}{x}^{m}}{1-\mult{\phi}{x}}\\
		&=\sum_{x\in\fixpt{\phi}}\frac{\tr{d\phi^{\otimes m}(x)}{\Omega_{\proj_K}^{\otimes m}(x)}}{\det (1-d\phi(x)\colon \Omega_{\proj_K}(x))}.
	\end{align*}
Summing up, we conclude
\[T_m(\phi)=\sum_{i=0}^{1}(-1)^{i}\tr{D_m^i(\phi)}{\cohomo{i}{m}}.\]
\end{proof}

%% Section 4.3. %%%%%%%%%%%%%%%%%%%%%%%%%%%%%%%%%%%%%%%%%%%%%%%%%%%%%%%%%%%%%%
\subsection{Lemmas on sheaf cohomologies}

We prepare some lemmas for cohomology of coherent sheaves on schemes. First we recall that there are an adjunction $f^{*}\dashv f_{*}$ for a morphism $f:X\to Y$ of schemes and there are two natural transformations $\epsilon_{f}: f^{*}f_{*} \to 1$ called the \textit{counit} and $\eta_{f}: 1\to f_{*}f^{*}$ called the \textit{unit}.  Although it is abusing notation, we use the same notation $\eta_{f}$ for corresponding sheaf homomorphism $\eta_{f}:\mathcal{F}\to f_{*}f^{*}\mathcal{F}$ for a coherent sheaf $\mathcal{F}$ on $Y$.

The next lemma shows that $\eta_{f}$ has the functoriality as following. 
\begin{lem}\label{lem_adjunction}
Let $X$,$Y$ and $Z$ be Noetherian schemes and $f:X \to Y$ and $g:Y \to Z$ be morphisms . Let $\mathcal{F}$ be a quasi-coherent sheaf on $Z$. Then the following diagram commutes.
	\begin{center}
		\begin{tikzcd}
			\mathcal{F} \arrow[rd,"\eta_{g\circ f}"]\arrow[d,"\eta_{g}"']& \\
			g_{*}g^{*}\mathcal{F} \arrow[r,"g_{*}\eta_{f}"']& g_{*}f_{*}f^{*}g^{*}\mathcal{F}. 
		\end{tikzcd}
	\end{center}
\end{lem}

\begin{proof}
Since the problem is local, we may assume that $X$,$Y$ and $Z$ are affine. We assume that $X=\spec{A}$, $Y=\spec{B}$, $Z=\spec{C}$ and $f:X \to Y$ and $g:Y \to Z$ come from $\varphi: B\to A$ and $\psi: C\to B$, respectively. We use the notation ${}_{C}N$ (\textit{resp}. ${}_{C}L$) for $B$-module $N$(\textit{resp}. $A$-module $L$) if we regard $N$ (\textit{resp}. $L$) as $C$-module via $\psi: C\to B$ (\textit{resp}. $\phi\circ\psi:C\to A$). 
Then there exsists a $C$ module $M$ such that $\mathcal{F}=\tilde{M}$ and we have $g_{*}g^{*}\mathcal{F}= {}_{C}(M\otimes_{C}B)^{\sim}$, $g_{*}f_{*}f^{*}g^{*}\mathcal{F}={}_{C}(M\otimes_{C}A)^{\sim}={}_{C}((M\otimes_{C}B)\otimes_B A)^{\sim}$. So the desired commutative diagram comes from the following diagrams.
\begin{center}
	\begin{tikzcd}
		M \arrow[d] \arrow[r]           & {}_{C}(M\otimes_{C}A) \arrow[d, no head, equal] & m \arrow[d, maps to] \arrow[r, maps to] & m\otimes 1 \arrow[d, maps to] \\
		{}_{C}(M\otimes_{C}B) \arrow[r] & {}_{C}((M\otimes_CB)\otimes_BA),                      & m\otimes 1 \arrow[r, maps to]           & (m\otimes 1)\otimes 1 .       
	\end{tikzcd}
\end{center}
\end{proof}

The next lemma shows that the differential $df$ has the functoriality as following. 
\begin{lem}\label{lem_differential}
Let $X$,$Y$ and $Z$ be Noetherian schemes over $K$ and $f:X \to Y$ and $g:Y \to Z$ be morphisms  over $K$. Then the following diagram commutes.
	\begin{center}
		\begin{tikzcd}
		f^{*}g^{*}\Omega_{Z} \arrow[rd,"d(g\circ f)"]\arrow[d,"f^{*}dg"']& \\
		f^{*}\Omega_{Y} \arrow[r,"df"']& \Omega_X.
		\end{tikzcd}
	\end{center}
\end{lem}

\begin{proof}
Since the problem is local, we may assume that $X$,$Y$ and $Z$ are affine. We assume that $X=\spec{A}$, $Y=\spec{B}$, $Z=\spec{C}$ and $f:X \to Y$ and $g:Y \to Z$ come from $\varphi: B\to A$ and $\psi: C\to B$, respectively. Then $df:f^{*}\Omega_{Y}\to\Omega_{X}$, $dg:g^{*}\Omega_{Z}\to\Omega_{Y}$ and $d(g\circ f):f^{*}g^{*}\Omega_{Z}\to\Omega_{X}$ correspond to $\Omega_B\otimes_{B}A\to \Omega_A; db\otimes a\mapsto ad(\varphi(b))$, $\Omega_C\otimes_{C}B\to \Omega_B; dc\otimes b\mapsto bd(\psi(c))$ and $\Omega_C\otimes_{C}A\to \Omega_A; dc\otimes a\mapsto ad(\varphi\psi(c))$, respectively. So the desired commutative diagram comes from the following diagrams.
\begin{center}
	\begin{tikzcd}
		\Omega_{C}\otimes_CA \arrow[d] \arrow[rd] &          & dc\otimes 1 \arrow[d, maps to] \arrow[rd, maps to] &                            \\
		\Omega_{B}\otimes_BA \arrow[r]            & \Omega_A, & d(\psi(c))\otimes 1 \arrow[r, maps to]             & d(\varphi\psi(c))\otimes 1.
	\end{tikzcd}
\end{center}
\end{proof}

\begin{lem}\label{comparison}
Let $\mathcal{A}$ be an abelian category and $f:A\to B$ be a morphism in $\mathcal{A}$.  Let $0\to A\to I^{\bullet}$ and $0\to B\to J^{\bullet}$ be complexes in $\mathcal{A}$. If each $J^{n}$ is injective, and if $0\to A\to I^{\bullet}$ is exact, then there exists a chain map $I^{\bullet}\to J^{\bullet}$ making the following diagram commute.
	\begin{center}
		\begin{tikzcd}
			0 \arrow[r] & A \arrow[r] \arrow[d, "f"] & I^{0} \arrow[r] \arrow[d] & I^{1} \arrow[r] \arrow[d] & {\cdots} \\
			0 \arrow[r] & B \arrow[r]                & J^{0} \arrow[r]           & J^{1} \arrow[r]           & {\cdots}.
		\end{tikzcd}
	\end{center}
Moreover, any two such chain maps are homotopic.
\end{lem}

\begin{proof}
See \cite[Theorem 6.16]{Rotman}.
\end{proof}

We review how the map $H^{p}(Y,f_{*}f^{*}\mathcal{F})\to H^{p}(X,f^{*}\mathcal{F})$ is constructed in general. Let $0\to f^{*}\mathcal{F}\to J^{\bullet}$ and $0\to f_{*}f^{*}\mathcal{F}\to I^{\bullet}$ be injective resolutions of $f^{*}\mathcal{F}$ and $f_{*}f^{*}\mathcal{F}$, respectively. Applying the functor $f_{*}$, we obtain a complex $0\to f_{*}f^{*}\mathcal{F}\to f_{*}J^{\bullet}$. Note that the complex $f_{*}J^{\bullet}$ consists of injective sheaves. By Lemma \ref{comparison}, we obtain a chain map $I^{\bullet}\to f_{*}J^{\bullet}$ and this chain map gives the map $H^{p}(Y,f_{*}f^{*}\mathcal{F})\to H^{p}(X,f^{*}\mathcal{F})$.

The next lemma shows that the map $H^{p}(Y,f_{*}f^{*}\mathcal{F})\to H^{p}(X,f^{*}\mathcal{F})$ has the functoriality as following. 
\begin{lem}\label{functoriality_of_pull-back}
\begin{itemize}
\item[$(1)$]Let $X$ and $Y$ be Noetherian schemes and $f:X\to Y$ be a morpshim. Let $\mathcal{F}$ and  $\mathcal{G}$ be sheaves of $\mathcal{O}_{Y}$ module and $\varphi : \mathcal{F}\to \mathcal{G}$ be a morphism of $\mathcal{O}_{Y}$ module. Then the following diagram commutes.
	\begin{center}
		\begin{tikzcd}
			{H^{p}(Y,f_{*}f^{*}\mathcal{F})} \arrow[d] \arrow[r] & {H^{p}(X,f^{*}\mathcal{F})} \arrow[d] \\
			{H^{p}(Y,f_{*}f^{*}\mathcal{G})} \arrow[r]           & {H^{p}(X,f^{*}\mathcal{G})},
		\end{tikzcd}
	\end{center}
	where $H^{p}(Y,f_{*}f^{*}\mathcal{F})\to H^{p}(Y,f_{*}f^{*}\mathcal{G})$ and $H^{p}(X,f^{*}\mathcal{F})\to H^{p}(X,f^{*}\mathcal{G})$ are induced by $f_{*}f^{*}\varphi : f_{*}f^{*}\mathcal{F}\to f_{*}f^{*}\mathcal{G}$ and $f^{*}\varphi: f^{*}\mathcal{F}\to f^{*}\mathcal{G}$, respectively.
\item[$(2)$]Let $X$,$Y$ and $Z$ be Noetherian schemes and let $f:X \to Y$ and $g:Y \to Z$ be morphisms. Let $\mathcal{F}$ be a  sheaf of $\mathcal{O}_{Z}$ module. Then the following diagram commutes.
	\begin{center}
		\begin{tikzcd}
			{H^{p}(Z,g_{*}f_{*}f^{*}g^{*}\mathcal{F})} \arrow[d] \arrow[rd] &                                  \\
			{H^{p}(Y,f_{*}f^{*}g^{*}\mathcal{G})} \arrow[r]                 & {H^{p}(X,f^{*}g^{*}\mathcal{F})}.
		\end{tikzcd}
	\end{center}
\end{itemize}
\end{lem}

\begin{proof}
For the first claim, let $0\to f^{*}\mathcal{F}\to I^{\bullet}$ and $0\to f^{*}\mathcal{G}\to J^{\bullet}$ be injective resolutions of $f^{*}\mathcal{F}$ and $f^{*}\mathcal{G}$, respectively. Then we obtain complexes $0\to f_{*}f^{*}\mathcal{F}\to f_{*}I^{\bullet}$ and $0\to f_{*}f^{*}\mathcal{G}\to f_{*}I^{\bullet}$.  By Lemma \ref{comparison}, we have a chain map $I^{\bullet}\to J^{\bullet}$ which makes the following diagram commute.
	\begin{center}
		\begin{tikzcd}
			0 \arrow[r] & f^{*}\mathcal{F} \arrow[r] \arrow[d, "f^{*}\varphi"] & I^{0} \arrow[r] \arrow[d] & I^{1} \arrow[r] \arrow[d] & {\cdots} \\
			0 \arrow[r] & f^{*}\mathcal{G} \arrow[r]                           & J^{0} \arrow[r]           & J^{1} \arrow[r]           & {\cdots}.
		\end{tikzcd}
	\end{center}
Applying the functor $f_{*}$, we obtain two complexes $0\to f_{*}f^{*}\mathcal{F}\to f_{*}I^{\bullet}$ and $0\to f_{*}f^{*}\mathcal{G}\to f_{*}J^{\bullet}$ and a chain map $f_{*}I^{\bullet}\to f_{*}J^{\bullet}$. Note that $f_{*}I^{\bullet}$ and $f_{*}J^{\bullet}$ consist of injective sheaves. Let $0\to f_{*}f^{*}\mathcal{F}\to I^{'\bullet}$ and $0\to f_{*}f^{*}\mathcal{G}\to J^{'\bullet}$ be injective resolutions of $f_{*}f^{*}\mathcal{F}$ and $f_{*}f^{*}\mathcal{G}$, respectively. By Lemma \ref{comparison}, we have chain maps $I^{'\bullet}\to J^{'\bullet}$, $I^{'\bullet}\to f_{*}I^{\bullet}$, and $J^{'\bullet}\to f_{*}J^{\bullet}$. Moreover, since the chain map $I^{'\bullet}\to f_{*}J^{\bullet}$ is unique up to homotopy, the following diagram commutes up to homotopy. 
	\begin{center}
		\begin{tikzcd}
			I^{'\bullet} \arrow[d] \arrow[r] & f_{*}I^{\bullet} \arrow[d] \\
			J^{'\bullet} \arrow[r]            & f_{*}J^{\bullet}.          
		\end{tikzcd}
	\end{center}
Therefore, the following diagram commutes.
\begin{center}
		\begin{tikzcd}
			{H^{p}(Y,f_{*}f^{*}\mathcal{F})} \arrow[d] \arrow[r] & {H^{p}(X,f^{*}\mathcal{F})} \arrow[d] \\
			{H^{p}(Y,f_{*}f^{*}\mathcal{G})} \arrow[r]           & {H^{p}(X,f^{*}\mathcal{G})}.
		\end{tikzcd}
	\end{center}
Next we prove the second claim. Let $0\to f^{*}g^{*}\mathcal{F}\to K^{\bullet}$ be an injective resolution of $f^{*}g^{*}\mathcal{F}$. Applying the functor $f_{*}$, we have a complex $0\to f_{*}f^{*}g^{*}\mathcal{F}\to f_{*}K^{\bullet}$. Let $0\to f_{*}f^{*}g^{*}\mathcal{F}\to J^{\bullet}$ be an injective resolution of $f_{*}f^{*}g^{*}\mathcal{F}$. By Lemma \ref{comparison}, we have a chain map $J^{\bullet}\to f_{*}K^{\bullet}$ which makes the following diagram commute.
	\begin{center}
		\begin{tikzcd}
			0 \arrow[r] & f_{*}f^{*}g^{*}\mathcal{F} \arrow[r] \arrow[d, equal] & J^{0} \arrow[r] \arrow[d] & J^{1} \arrow[r] \arrow[d] & {\cdots} \\
			0 \arrow[r] & f_{*}f^{*}g^{*}\mathcal{F} \arrow[r]                           & f_{*}K^{0} \arrow[r]           & f_{*}K^{1} \arrow[r]           & {\cdots}.
		\end{tikzcd}
	\end{center}
For an injective resolution $0\to g_{*}f^{*}f^{*}g^{*}\mathcal{F}\to I^{\bullet}$ of $g_{*}f^{*}f^{*}g^{*}\mathcal{F}$, we obtain chain maps $I^{\bullet}\to g_{*}J^{\bullet}$ and $I^{\bullet}\to g_{*}f_{*}K^{\bullet}$. Since the chain map $I^{\bullet}\to g_{*}f_{*}K^{\bullet}$ is unique up to homotopy, the following diagram commutes up to homotopy.
	\begin{center}
		\begin{tikzcd}
			I^{\bullet} \arrow[d] \arrow[rd] &                       \\
			g_{*}J^{\bullet} \arrow[r]       & g_{*}f_{*}K^{\bullet}.
		\end{tikzcd}
	\end{center}
Therefore, the following diagram commutes.
	\begin{center}
		\begin{tikzcd}
			{H^{p}(Z,g_{*}f_{*}f^{*}g^{*}\mathcal{F})} \arrow[d] \arrow[rd] &                                  \\
			{H^{p}(Y,f_{*}f^{*}g^{*}\mathcal{G})} \arrow[r]                 & {H^{p}(X,f^{*}g^{*}\mathcal{F})}.
		\end{tikzcd}
	\end{center}
\end{proof}

\begin{prop}\label{lem_cohomology}
Let $X$ be a Noetherian scheme over $K$ and let $f:X\to X$ and $g:X\to X$ be endomorphisms over $K$. Then for all $p,m \in \Z_{\ge 0}$ the following diagram commutes.
	\begin{center}
		\begin{tikzcd}
			H^{p}(X,\sheaf) \arrow[d, "D_m^p(g)"'] \arrow[rd, "D_m^p(g\circ f)"]& \\
			H^{p}(X,\sheaf) \arrow[r, "D_m^p(f)"'] & H^{p}(X,\sheaf),
		\end{tikzcd}
	\end{center}	
where $D_m^p(f): H^{p}(X,\sheaf)\to H^{p}(X,\sheaf)$ is the map associated to the sheaf homomorphism $(df)^{\otimes m}:f^{*}\Omega_{X}^{\otimes m}\to \Omega_{X}^{\otimes m}$.
\end{prop}

\begin{proof}
By combining Lemma \ref{lem_adjunction}, Lemma \ref{lem_differential}, and the naturality of unit $\eta_{f}: 1\to f_{*}f^{*}$, we have following three commutative diagrams.

\begin{center}
	\begin{tikzcd}
	\sheaf \arrow[d, "\eta_{g}"'] \arrow[rd, "\eta_{g\circ f}"] &                                                      &                                                                          &             \\
	g_{*}g^{*}\sheaf \arrow[r, "g_{*}\eta_{f}"']                    & g_{*}f_{*}f^{*}g^{*}\sheaf, &                                                                          &             \\
	g^{*}\sheaf \arrow[d, "dg^{\otimes m}"'] \arrow[r, "\eta_{f}"]              & f_{*}f^{*}g^{*}\sheaf \arrow[d, "f_{*}f^{*}dg^{\otimes m}"] & f^{*}g^{*}\sheaf \arrow[d, "f^{*}dg^{\otimes m}"'] \arrow[rd, "d{(g \circ f)^{\otimes m}}"] &             \\
	\sheaf \arrow[r, "\eta_{f}"']                                   & f_{*}f^{*}\sheaf, & f^{*}\sheaf \arrow[r, "df^{\otimes m}"']                                        & \sheaf.
	\end{tikzcd}
\end{center}

Taking cohomology $H^{p}(X,-)$ and using Lemma \ref{functoriality_of_pull-back}, we obtain

\begin{center}
	\begin{tikzcd}
	{H^{p}(X,\Omega_{X}^{\otimes m})} \arrow[d] \arrow[rd] \arrow[dd, "D_{m}^{p}(g)"', bend right=67] \arrow[rrdd, "D_{m}^{p}(g\circ f)", bend left] &                                                                  &                                   \\
	{H^{p}(X,g^{*}\Omega_{X}^{\otimes m})} \arrow[d] \arrow[r]                                                                     & {H^{p}(X,g^{*}f^{*}\Omega_{X}^{\otimes m})} \arrow[d] \arrow[rd] &                                   \\
	{H^{p}(X,\Omega_{X}^{\otimes m})} \arrow[r] \arrow[rr, "D_{m}^{p}(f)"', bend right]                                                      & {H^{p}(X,f^{*}\Omega_{X}^{\otimes m})} \arrow[r]                 & {H^{p}(X,\Omega_{X}^{\otimes m})}.
	\end{tikzcd}
\end{center}.
\end{proof}

%% Section 4.3. %%%%%%%%%%%%%%%%%%%%%%%%%%%%%%%%%%%%%%%%%%%%%%%%%%%%%%%%%%%%%%
\subsection{Conclusion of the proof}\label{conclution}
Now we are ready to prove Theorem \ref{maintheorem}.
\begin{proof}
Using Lemma \ref{lem_localvsgrobal}, we obtain 
	\begin{align*}
		T_m(\phi^n)&=\sum_{i=0}^{1}(-1)^i\tr{D_m^i(\phi^n)}{\cohomo{i}{m}}.
	\end{align*}
Therefore, 
\begin{align*}
\lzeta{m}{\phi}=\exp\left( \sum_{n=1}^{\infty}\frac{t^n}{n} \sum_{i=0}^{1}(-1)^{i}\mathrm{tr}D_m^i(\phi^n)\right).
\end{align*}

Using Proposition \ref{lem_cohomology}, we have $D_m^i(\phi^n)=D_m^i(\phi)^n$. Note that  $D_m^i(\phi)$ is a $K$-linear operator acting on the finite dimensional $K$-vector space $\cohomo{i}{m}$.  Therefore, we obtain
\begin{align*}
&\exp\left( \sum_{n=1}^{\infty}\frac{t^n}{n} \sum_{i=0}^{1}(-1)^{i}\mathrm{tr}D_m^i(\phi^n)\right)\\
=&\exp\left( \sum_{n=1}^{\infty}\frac{t^n}{n} \sum_{i=0}^{1}(-1)^{i}\mathrm{tr}D_m^i(\phi)^n\right)\\
=&\exp\left( \sum_{n=1}^{\infty}\frac{t^n}{n} \mathrm{tr}D_m^0(\phi)^n\right)\exp\left( \sum_{n=1}^{\infty}\frac{t^n}{n} \mathrm{tr}D_m^1(\phi)^n\right)^{-1}\\
=&\frac{\det (1-tD_m^1(\phi))}{\det (1-tD_{m}^0(\phi))}\in K(t).
\end{align*}

Summing up, we conclude
\[\lzeta{m}{\phi}=\dfrac{\det (1-tD_m^1(\phi))}{\det (1-tD_{m}^0(\phi))} .\]

\end{proof}

%% Section 5. %%%%%%%%%%%%%%%%%%%%%%%%%%%%%%%%%%%%%%%%%%%%%%%%%%%%%%%%%%%%%%

\section{Examples}\label{examples}

In this section, we construct some examples of completely transversal rational functions and calculate an explicit form of $\dzeta{m}{\phi}$ using Theorem \ref{maintheorem}. We consider the case $K=\C$ since we use the results on complex dynamics.

\begin{defin}
Let $x$ be a periodic point with minimal period $n$ for a rational function $\phi\in \C(z)$, and $\lambda=\mult{\phi^n}{x}$ be the multiplier of $\phi$ at $x$. Then $x$ is:
\begin{itemize}
\item[$(1)$]\textit{attracting} if $|\lambda|<1$;
\item[$(2)$]\textit{repelling} if $|\lambda|>1$;
\item[$(3)$]\textit{rationally indifferent} if $\lambda$ is a root of unity;
\item[$(4)$]\textit{irrationally indifferent} if $|\lambda|=1$, but $\lambda$ is not a root of unity.
\end{itemize}
\end{defin}

\begin{defin}
Let $\phi\in\C(z)$ be a rational function and $x\in\per{n}{\phi}$.  For a point $c\in\proj_{\C}$, we say that $c$ \textit{is attracted to the orbit of} $x$ if there exists $i\in\{0,1,\dots,n-1\}$ such that $\displaystyle \lim_{m\to\infty}\phi^{mn}(c)=\phi^{i}(x)$ with respect to the classical topology of $\proj_{\C}$.
\end{defin}

\begin{thm}\label{att_rat_ind}
Let $\phi\in\C(z)$ be a rational function. If $x$ is a periodic point of $\phi$ and $x$ is attracting or rationally indifferent, then there exists a critical point $c$ of $\phi$ which is attracted to the orbit of $x$.
\end{thm}

\begin{proof}
See \cite[Theorem 9.3.1. and Theorem 9.3.2.]{Beardon}.
\end{proof}

\begin{rem}\label{crit}
If $\phi(z)=F(z)/G(z)$ in lowest terms, the \textit{degree} of $\phi$ is $\deg \phi=\mathrm{max}\{\deg F,\deg G\}$. By Riemann-Hurwitz formula, $\phi$ has at most $2d-2$ critical points. See \cite[Section 1.2]{Silverman}.
\end{rem}

Using Theorem \ref{att_rat_ind}, we can find a sufficient condition for the complete transversality.
\begin{cor}\label{sufficient}
Let $\phi\in\C(z)$ be a rational function of degree $d$. If $\phi$ has $2d-2$ attracting periodic points whose orbits are pairwise distinct, then $\phi$ is completely transversal. 
\end{cor}

\begin{proof}
By the definition of the complete transversality, $\phi$ is completely transversal if $\phi$ has no rationally indifferent periodic points. By Theorem \ref{att_rat_ind} and \ref{crit}, $\phi$ has no rationally indifferent periodic points if $\phi$ has $2d-2$ attracting periodic points whose orbits are pairwise distinct.
\end{proof}

We can construct completely transversal rational functions using Corollary \ref{sufficient}.
\begin{ex}
Let $\lambda_0,\lambda_{\infty}\in \C$ be complex numbers such that $|\lambda_0|<1$ and $|\lambda_{\infty}|<1$. We define $\phi\in\C(z)$ by 
\[\phi(z)=\frac{z^2+\lambda_0 z}{\lambda_{\infty} z+1}\in\C(z).\]
Then we have
\[\fixpt{\phi}=\left\{0,\infty,\alpha=\frac{1-\lambda_{0}}{1-\lambda_{\infty}}\right\}.\]
The multipliers of $\phi$ at $0$ and $\infty$ are $\mult{\phi}{0}=\lambda_0$ and $\mult{\phi}{\infty}=\lambda_{\infty}$, respectively. Since $|\lambda_0|<1$ and $|\lambda_{\infty}|<1$, $\phi$ has two attracting fixed points. On the other hand, $\phi$ has at most $2=2d-2$ critical points since $d=\deg \phi =2$.  Therefore, $\phi$ is completely transversal.
\end{ex}

\begin{rem}
All nonconstant polynomials have a fixed point at $\infty$ and the multiplier at $\infty$ is $0$. Therefore, if $\psi\in\C(z)$ is conjugate to a polynomial, then $\psi$ has a fixed point with multiplier $0$. Since $\mult{\phi}{0}=\lambda_0$, $\mult{\phi}{\infty}=\lambda_{\infty}$, and $\mult{\phi}{\alpha}=\dfrac{2-\lambda_{0}-\lambda_{\infty}}{1-\lambda_{0}\lambda_{\infty}}$, $\phi$ is not conjugate to any polynomial if $\lambda_{0}\lambda_{\infty}\neq 0$. 
\end{rem}

Next, we calculate the dynamical zeta function $\dzeta{1}{\phi}$ using the formula
\[\dzeta{1}{\phi}=\frac{\dyndet{1}}{\dyndet{2}}.\]

\begin{ex}
We use {\v{C}}ech cohomology to compute the linear maps $D_{1}^{1}(\phi)$ and $D_{2}^{1}(\phi)$ explicitly.
We define $\Fz(z),\Fi(z)\in \C[z]$, and $G_0(w)\in\C[w]$ by $\Fz(z)=z+\lamz$, $\Fi(z)=\lami z +1$, and $G_0(w)=\lamz w +1$, respectively. Note that $\phi(z)=z\Fz/\Fi$ and $G_0(1/z)=\Fz(z)/z$. We take open coverings $\mathcal{U}=\{U_0, U_1\}$ and $\mathcal{V}=\{V_0, V_1\}$ of $\proj$, where $U_0=\proj\setminus\{\infty\}=\spe{\C[z]}$, $U_1=\proj\setminus\{0\}=\spe{\C[w]}$, $V_0=\phi^{-1}(U_0)=\proj\setminus\{\infty, -\lami^{-1}\}=\spe{\C[z,\Fi^{-1}]}$, and $V_1=\phi^{-1}(U_1)=\proj\setminus\{0, -\lamz\}=\spe{\C[w,G_0^{-1}]}$. Then the differential induces the map $\check{H}^{1}(\mathcal{U},\Omega_{\proj}^{\otimes m})\to \check{H}^{1}(\mathcal{V},\Omega_{\proj}^{\otimes m})$ and this map is identified with $D_{m}^{1}(\phi)$ via the isomrphisms $H^{1}(\proj,\Omega_{\proj}^{\otimes m})\cong \check{H}^{1}(\mathcal{U},\Omega_{\proj}^{\otimes m})$ and $H^{1}(\proj,\Omega_{\proj}^{\otimes m})\cong \check{H}^{1}(\mathcal{V},\Omega_{\proj}^{\otimes m})$.

$\check{H}^{1}(\mathcal{U},\Omega_{\proj}^{\otimes m})$ is the first cohomology of the complex
\[\C[z](dz)^{\otimes m}\times \C[w](dw)^{\otimes m}\to \C[z^{\pm 1}]\dform{m}.\]
$\check{H}^{1}(\mathcal{V},\Omega_{\proj}^{\otimes m})$ is the first cohomology of the complex
\[\C[z,\Fi^{-1}](dz)^{\otimes m}\times \C[w,G_{0}^{-1}](dw)^{\otimes m}\to \C[z^{\pm 1},\Fz^
{-1},\Fi^{-1}]\dform{m}.\]
Note that $\dfrac{z^m a(z)}{\Fi^i}\left(\dfrac{dz}{z}\right)^{\otimes m}=0$ and $\dfrac{b(1/z)}{z^{m-j}\Fz^j}\left(\dfrac{dz}{z}\right)^{\otimes m}=0$ in $\check{H}^{1}(\mathcal{V},\Omega_{\proj}^{\otimes m})$ for all $a(z)\in\C[z]$ and $b(w)\in\C[w]$ since they are the image of $\dfrac{a(z)}{\Fi^i}(dz)^{\otimes m}$ and $\dfrac{b(w)}{G_0^j}(dw)^{\otimes m}$, respectively.
Both $\check{H}^{1}(\mathcal{U},\Omega_{\proj}^{\otimes m})$ and $\check{H}^{1}(\mathcal{V},\Omega_{\proj}^{\otimes m})$ have a $\C$-basis
\[\left\{ z^i \dform{m} \colon |i|<m\right \}.\]
The image of $z^i\left(\dfrac{dz}{z}\right)^{\otimes m}$ by $(d\phi)^{\otimes m}$ is
\begin{align*}
(d\phi)^{\otimes m}\left(z^{i}\dform{m}\right)&=\phi^{i}\frac{d\phi}{\phi}\\
&=z^i\frac{\Fz^{i}}{\Fi^{i}}\left(1+\frac{(1-\lamz\lami)z}{\Fz\Fi}\right)^{m}\dform{m}\\
&=\sum_{j=0}^{m}\binom{m}{j}(1-\lamz\lami)^j\frac{z^{i+j}}{\Fz^{j-i} \Fi^{i+j}}\dform{m}.
\end{align*}
We can compute $\dfrac{z^{i+j}}{\Fz^{j-i} \Fi^{i+j}}\left(\dfrac{dz}{z}\right)^{\otimes m}$ by
\[ \frac{\lamz}{\Fz}=-\sum_{n=1}^{m-1}\left(\frac{-\lamz}{z}\right)^n-\frac{(-\lamz)^m}{z^{m-1}\Fz},\]
\[ \frac{1}{\Fi}=\sum_{n=0}^{m-1}(-\lami z)^n+\frac{(-\lami z)^{m}}{\Fi}, \,\text{and}\]
\[ \frac{(1-\lamz\lami)z}{\Fz\Fi}=\frac{1}{\Fi}-\frac{\lamz}{\Fz}.\]
For example, if $m=1$ then $\dfrac{1}{\Fi}\dfrac{dz}{z}=\dfrac{dz}{z}$ and $\dfrac{\lamz}{\Fz}\dfrac{dz}{z}=0$ in $\check{H}^{1}(\mathcal{V},\Omega_{\proj})$. Therefore,
\begin{align*}
d\phi\left(\frac{dz}{z}\right)&=\sum_{j=0}^{1}\binom{1}{j}(1-\lamz\lami)^j\frac{z^{j}}{\Fz^{j} \Fi^{j}}\frac{dz}{z}\\
&=\frac{dz}{z}+\frac{(1-\lamz\lami)z}{\Fz\Fi}\frac{dz}{z}\\
&=\frac{dz}{z}+\left(\frac{1}{\Fi}-\frac{\lamz}{\Fz}\right)\frac{dz}{z}\\
&=2\frac{dz}{z}.
\end{align*}
The characteristic polynomial of $D_{1}^{1}(\phi)$ is $\det(1-tD_{1}^{1}(\phi))=1-2t$.

A similar but complicated computation shows that the representation matrix of $D_{2}^{1}(\phi)$ with respect to the basis $\{z^i(dz/z)^{\otimes 2} \colon |i|<2\}$ is 
\begin{align*}
\begin{pmatrix}
-\dfrac{\lamz\lami^{2}}{1-\lamz\lami}	& -2\lamz\dfrac{2-\lamz\lami}{1-\lamz\lami}	& -\dfrac{\lamz^{3}}{1-\lamz\lami} \\[8pt]
\dfrac{\lami^{2}}{1-\lamz\lami}			& 2\dfrac{2-\lamz\lami}{1-\lamz\lami}			& \dfrac{\lamz^{2}}{1-\lamz\lami} \\
-\dfrac{\lamz^{2}\lami}{1-\lamz\lami}	& -2\lami\dfrac{2-\lamz\lami}{1-\lamz\lami}	&
 -\dfrac{\lamz^{2}\lami}{1-\lamz\lami} \\
\end{pmatrix}
=
\begin{pmatrix}
\lamz A 	& \lamz B 	& \lamz C\\
-A 			& -B 			& -C\\
\lami A 	& \lami B 		& \lami C\\
\end{pmatrix},
\end{align*}
where 
\[A=-\frac{\lami^2}{1-\lamz\lami},\, B=-2 \dfrac{2-\lamz\lami}{1-\lamz\lami}, \text{and}\,  C=-\dfrac{\lamz^{2}}{1-\lamz\lami}.\]
Therefore, the characteristic polynomial of $D_{2}^{1}(\phi)$ is
\begin{align*}
\det(1-tD_{2}^{1}(\phi))&=1+(B-\lamz A- \lami C)t\\
&=1-\left(2+\lamz+\lami+\frac{2-\lamz-\lami}{1-\lamz\lami}\right)t\\
&=1-(2+\sigma_{1}(\phi))t,
\end{align*}
where $\sigma_{1}(\phi)=\mult{\phi}{0}+\mult{\phi}{\infty}+\mult{\phi}{\alpha}$ is the sum of the multipliers at fixed points of $\phi$.

Summing up, we obtain an explicit formula for $\dzeta{1}{\phi}$.
\begin{align*}
\dzeta{1}{\phi}&=\frac{\dyndet{1}}{\dyndet{2}}\\
&=\frac{1-2t}{1-(2+\sigma_1(\phi))t}.
\end{align*}
\end{ex}

\bibliographystyle{plain}
\bibliography{dynamicalzeta_1}

\end{document}